\documentclass[oneside,10pt]{amsart}          
\usepackage[b5paper]{geometry}	    
\usepackage{amsfonts,amsmath,latexsym,amssymb} 
\usepackage{amsthm}                
\usepackage{mathrsfs,upref}         
\usepackage{mathptmx}		    
\usepackage{amsmath}
\usepackage{enumerate}

\newtheorem{theorem}{Theorem}[section]
\newtheorem{lemma}[theorem]{Lemma}
\newtheorem{proposition}[theorem]{Proposition}
\newtheorem{corollary}[theorem]{Corollary}
\newtheorem{conjecture}[theorem]{Conjecture}

\theoremstyle{definition}

\newtheorem{example}[theorem]{Example}
\numberwithin{equation}{section}

\def\nk0t{\|\tilde k_0^\theta\|^{-2}}
\def\kda{K^2_\alpha}
\def\kdt{K^2_\theta}

\def\b1{\mathcal{B}_1(\kdt,\kda)}
\begin{document}
\title[Toeplitz operators and conjugations]{Characterization of truncated Toeplitz operators\\ by conjugations}
\thanks{The research of the first and the third authors was financed by the Ministry of Science and Higher Education of the Republic of Poland}

\author{Kamila Kli\'s-Garlicka, Bartosz \L anucha and Marek Ptak}

\address{Kamila Kli\'s-Garlicka, Institute of Mathematics,
University of Agriculture, Balicka 253c,30-198 Krakow, Poland}
\email{rmklis@cyfronet.pl}

\address{Bartosz \L anucha, Department of Mathematics,  Maria Curie-Sk\l odowska University, Maria Curie-Sk\l o\-dow\-ska Square 1, 20-031 Lublin, Poland}
\email{bartosz.lanucha@poczta.umcs.lublin.pl}

\address{Marek Ptak, Institute of Mathematics,
University of Agriculture, Balicka 253c, 30-198 Krakow, Poland  and
Institute of Mathematics, Pedagogical University, ul.
Podcho\-r\c{a}\.{z}ych 2, 30-084 Krak\'ow, Poland}
\email{rmptak@cyfronet.pl}


\date{12 December 2016}
\keywords{conjugation, $C$-symmetric operators, model space, truncated Toeplitz operator, Blaschke product, inner function}

\subjclass{Primary 47B35, Secondary 47B32, 30D20}

\begin{abstract}
Truncated Toeplitz operators are C--symmetric with respect to the canonical conjugation given on an appropriate model space. However, by considering only one conjugation one cannot characterize truncated  Toeplitz operators.  It will be proved, for some classes of inner functions and the model spaces connected with them, that if an operator on a model space is C--symmetric for a certain family of conjugations in the model space, then is has to be truncated Toeplitz. A characterization of classical Toeplitz operators is also presented in terms of conjugations.
\end{abstract}
\maketitle

\section{Introduction}
Let $\mathcal{H}$ denote a complex Hilbert space. Denote by $L(\mathcal{H})$ the algebra of all bounded linear operators on $\mathcal{H}$.  A {\it conjugation} is an antilinear involution $C\colon \mathcal{H}\to\mathcal{H}$ such that $\langle Cf,Cg\rangle=\langle g,f\rangle$ for all $f, g\in \mathcal{H}$. An operator $A\in L(\mathcal{H})$ is called {\it C--symmetric} if $CAC=A^*$.

Let $\mathbb{D}$ denote the open unit disk, let $\mathbb{T}=\partial\mathbb{D}$ denote the unit circle and let $m$ be the normalized Lebesgue measure on $\mathbb{T}$. Denote by $L^2$ the space $L^2(\mathbb{T},m)$ and by $L^\infty=L^\infty(\mathbb{T},m)$. Recall that a  classical {\it Toeplitz operator} $T_\varphi$ with a symbol $\varphi\in L^\infty$ on the Hardy space $H^2$ is given by the formula
$$T_\varphi f=P(\varphi f) \text{ for } f\in H^2,$$
where $P\colon L^2\to H^2$ is the orthogonal projection. Denote by $\mathcal{T}$ the set of all Toeplitz operators, i.e., $\mathcal{T}=\{T_\varphi : \varphi \in L^\infty\}$.

Let $\theta$ be a nonconstant inner function. Consider the so-called {\it model space} $\kdt=H^2\ominus\theta H^2$ and the orthogonal projection $P_\theta\colon L^2\to\kdt$.  A {\it truncated Toeplitz operator} $A^\theta_\varphi$ with a symbol $\varphi\in L^2$ is defined as
$$A^\theta_\varphi\colon D(A^\theta_\varphi)\subset K^2_\theta\to K^2_\theta; \quad A^\theta_\varphi f=P_\theta(\varphi f)$$
for $f\in D(A^\theta_\varphi)=\{f\in K^2_\theta : \varphi f\in L^2\}$. Denote by $\mathcal{T}(\theta)$ the set of all bounded truncated Toeplitz operators on $\kdt$.

The conjugation $C_\theta$ defined for $f\in L^2$ by the formula
$$C_\theta f(z)=\theta(z) \overline{zf(z)},\quad |z|=1,$$
is a very useful tool  in investigating Toeplitz operators. In fact, all truncated Toeplitz operators are $C_\theta$--symmetric \cite{GP}.

Truncated Toeplitz operators have been recently strongly investigated (see for instance \cite{Sarason,BCT, BCKP, CCJP1, CRW, GMR,  lanucha}). However, usually only one (canonical) conjugation was involved in analysis on these operators.
In this paper we suggest to consider a family of conjugations to study Toeplitz operators. In particular,
we give a characterization of the classical Toeplitz operators as well as some special cases of truncated Toeplitz operators  using conjugations.

It is easy to see that if $\theta=z^N$, then $\kdt=\mathbb{C}^N$. The natural conjugation $C_N=C_{z^N}$ in $\mathbb{C}^N$ can be expressed as $C_N(z_1,\dots, z_N)=(\bar z_N,\dots, \bar z_1)$. Note that a matrix $(a_{i,j})_{i,j=1,\dots,N}$ is $C_N$--symmetric if and only if it is symmetric with respect to the second diagonal, i.e.,
$$a_{i,j}=a_{N-j+1,N-i+1}\quad\text{for}\quad i,j=1,\dots,N.$$
 On the other hand, a finite matrix $(a_{i,j})_{i,j=1,\dots,N}$ is a Toeplitz matrix if and only if it has constant diagonals, that is,
 $$a_{i,j}=a_{k,l}\quad\text{if}\quad i-j=k-l.$$
 Hence, as D. Sarason in \cite{Sarason} observed, each $N\times N$ Toeplitz matrix is $C_N$--symmetric but the reverse implication is true only if $N\leqslant 2$. However, one can notice that for a given matrix $(a_{i,j})_{i,j=1,\dots,N}$, if the matrix is $C_n$--symmetric for every $n\leqslant N$, i.e.,
$$a_{i,j}=a_{n-j+1,n-i+1} \text{ for } n\leqslant N \text{ and } i,j=1,\dots n,$$
then the matrix $(a_{i,j})_{i,j=1,\dots,N}$ has to be Toeplitz. Corollary \ref{zn} gives a precise proof of this fact.  One can ask if a similar property can be obtained for other  inner functions than $\theta= z^N$.
Using known matrix descriptions \cite{CRW, GMR, lanucha} we obtained
the positive answer: for Blaschke product with  a single zero in Section 3, for a finite Blaschke product with distinct zeros  in
 Section 4 (the most demanding case), for an infinite Blaschke product with uniformly separated zeros in Section 5. For a general case  we put the conjecture in Section 6. However,  even for the simplest singular inner function $\theta(z)=\exp(\frac{z+1}{z-1})$  no similar description is known and to solve the conjecture probably a different approach is needed. In Section 2 we also give  similar characterization of the classical Toeplitz operators on the Hardy space in terms of conjugations.

\section{Characterization of Toeplitz operators by conjugations}

Let $\alpha$ and $\theta$ be two nonconstant inner functions. We say that $\alpha$ divides $\theta$ ($\alpha\leqslant\theta$) if $\bar\alpha\theta$ is an inner function. It is easy to verify that $K_{\alpha}^2\subset K_{\theta}^2$ for every $\alpha\leqslant\theta$.~It is known that truncated Toeplitz operators on $K^2_\theta$ are $C_\theta$--symmetric but this property does not characterize them, i.e., there are $C_\theta$--symmetric operators on $K^2_\theta$, which are not truncated Toeplitz (\cite{GP}, \cite[Lemma 2.1, Corollary on p.504]{Sarason}). Note however that $A_\varphi^\theta$ is $C_\alpha$--symmetric for every $\alpha\leqslant \theta$. Namely:

\begin{lemma}\label{obciecie}
Let $A_\varphi^\theta\colon K^2_\theta\to K^2_\theta$ be a truncated Toeplitz operator. For every $\alpha \leqslant \theta$ the operator $P_\alpha A_{\varphi|K^2_\alpha}^\theta$ is $C_\alpha$--symmetric.
\end{lemma}
\begin{proof}
Note that $P_\alpha A_{\varphi|K^2_\alpha}^\theta$ belongs to $\mathcal{T}(\alpha)$. Actually, $P_\alpha A_{\varphi|K^2_\alpha}^\theta=A_\varphi^\alpha$, hence it is $C_\alpha$--symmetric by \cite[Lemma 2.1]{Sarason}.
\end{proof}

A similar argument shows that if $A\in\mathcal{T}$, then $P_\alpha A_{|K^2_\alpha}$ is $C_\alpha$--symmetric for all inner functions $\alpha$. The latter can be used to characterize all Toeplitz operators on $H^2$:

\begin{theorem}\label{toep}
Let $A\in L(H^2)$. Then the following conditions are equivalent:
 \begin{enumerate}[(1)]
 \item $A\in \mathcal{T}$;
 \item  $C_\alpha A_\alpha C_\alpha = A^*_\alpha$ for all nonconstant inner functions $\alpha$, where $A_\alpha = P_\alpha A_{|K^2_\alpha}$;
 \item  $C_\alpha A_\alpha C_\alpha = A^*_\alpha$ for all $\alpha=z^n$, where $A_\alpha = P_\alpha A_{|K^2_\alpha}$.
  \end{enumerate}
\end{theorem}

\begin{proof}
The proof of the implication $(1)\Rightarrow (2)$ is similar to the proof of Lemma \ref{obciecie}. Since $(2)\Rightarrow (3)$ is obvious, we will prove now that $(3)\Rightarrow (1)$.

The equivalent condition for a bounded operator on $H^2$ to be Toeplitz is that it has to annihilate all rank-two operators of the form
$$t=z^m\otimes z^r-z^{m+1}\otimes z^{r+1}\quad\text{with}\quad m, r\geq 0,$$
 in the sense that $\operatorname{tr} (A t)=0$ (it follows form the well known Brown--Halmos characterization of Toeplitz operators  given in \cite{BH}). Each such operator can be obtained from $1\otimes z^k-z^l\otimes z^{k+l}$ or $ z^k \otimes 1-z^{k+l}\otimes z^{l}$, with $k, l\geq 0$. Hence our reasoning will be held only for such operators.

Fix $k, l\geq 0$ and let $\alpha=z^n$, $n=k+l+1$. Since
$$C_{\alpha}z^k=z^{n-k-1}=z^l\quad\text{and}\quad C_{\alpha}1=z^{n-1}=z^{k+l},$$
the $C_{\alpha}$--symmetry of $A_{\alpha}$ gives
 \begin{displaymath}
 \begin{split}
\operatorname{tr} (A (1\otimes z^k))&=\langle A 1,z^k \rangle=\langle A_\alpha 1,z^k \rangle=\langle C_\alpha z^k, C_\alpha A_\alpha 1\rangle\\
 &=\langle C_\alpha z^k, A^*_\alpha C_\alpha 1\rangle=\langle z^{l}, A^*_\alpha z^{k+l}\rangle
=\langle A_\alpha z^{l}, z^{k+l}\rangle=\operatorname{tr} (A (z^{l}\otimes z^{k+l})).
 \end{split}
\end{displaymath}
Similarly,
$$\operatorname{tr} (A ( z^k\otimes 1))=\operatorname{tr} (A(z^{k+l}\otimes z^{l})).$$
Therefore all operators of the form $1\otimes z^k-z^{l}\otimes z^{k+l}$, $z^k\otimes 1-z^{k+l}\otimes z^{l}$ for $k, l\geq 0$, are annihilated by $A$. Hence $A$ is Toeplitz.
\end{proof}
 From the previous proof we can obtain
\begin{corollary}\label{zn}
  Let $A\in L(K^2_{z^N})$, $N\in\mathbb{N}$. Then $A\in \mathcal{T}(z^N)$ if and only if for every $1\leqslant n\leqslant N$ the operator $A_n$ is $C_{z^n}$--symmetric, i.e., $C_{z^n}A_nC_{z^n}=A_n^*$, where $A_n=P_{n}A_{|K^2_{z^n}}$ and $P_n\colon K^2_{z^N}\to K^2_{z^n}$ is the orthogonal projection.
\end{corollary}

\section{The case of a Blaschke product with a single zero}

Let $\alpha$, $\theta$ be any nonconstant inner functions. We say that a unitary operator $U\colon K^2_\theta\to\kda$ defines a spatial isomorphism between $\mathcal{T}(\theta)$ and $\mathcal{T}(\alpha)$ if $U\mathcal{T}(\theta)U^{*}=\mathcal{T}(\alpha)$, that is, $A\in\mathcal{T}(\theta)$ if and only if $UAU^{*}\in\mathcal{T}(\alpha)$. If such $U$ exists, $\mathcal{T}(\theta)$ and $\mathcal{T}(\alpha)$ are said to be spatially isomorphic. The spatial isomorphism  between spaces of truncated Toeplitz operators is discussed in \cite[Chapter 13.7.4]{GMR}.

\begin{proposition}\label{spis}
Let $\alpha$, $\theta$ be any nonconstant inner functions. Let $U\colon K^2_\theta\to\kda$ be such that $U$ defines a spatial isomorphism between $\mathcal{T}(\theta)$ and $\mathcal{T}(\alpha)$. Then $UC_\theta=C_\alpha U$.
\end{proposition}
\begin{proof}

It is known \cite[Chapter 13.7.4]{GMR} that there are three basic types of unitary operators that define a spatial isomorphism between $\mathcal{T}(\theta)$ and $\mathcal{T}(\alpha)$. The requested intertwining property for one of those basic types is proved in \cite[Lemma 13.1]{Sarason}. The proof for two other types is similar. Since every $U\colon K^2_\theta\to\kda$ such that $U$ defines a spatial isomorphism between $\mathcal{T}(\theta)$ and $\mathcal{T}(\alpha)$, is a composition of at most three of those basic types of operators, it follows that $U$ also has this intertwining property.
\end{proof}

Let $a\in \mathbb{D}$ and $N\in\mathbb{N}$. Denote $b_a(z)=\frac{z-a}{1-\bar az}$.

\begin{proposition}\label{b1}
Let $A\in L(K^2_{b_a^N})$. Then $A\in\mathcal{T}(b_a^N)$ if and only if for every $1\leqslant n\leqslant N$ the operator $A_n$ is $C_{b_a^n}$--symmetric, i.e., $C_{b_a^n}A_nC_{b_a^n}=A_n^{*}$, where $A_n=P_nA_{|K^2_{b_a^n}}$ and $P_n\colon K^2_{b_a^N}\to K^2_{b_a^n}$ is the orthogonal projection.
\end{proposition}

\begin{proof}
The operator $U_{b_a}$ given by
$$U_{b_a}f(z)=\frac{\sqrt{1-|a|^2}}{1-\bar az}f\circ b_a(z)$$
defines a spatial isomorphism between $\mathbb{C}^n=K^2_{z^n}$ and $K^2_{b_a^n}$ for each $n=1,\dots, N$ (see \cite[chapter 13.7.4(i)]{GMR}). By Proposition \ref{spis}, $U_{b_a}$ intertwines the conjugations $C_{z^n}$ and $C_{b_a^n}$. Application of Corollary \ref{zn} finishes the proof.
\end{proof}

\section{The case of a finite Blaschke product with distinct zeros}

Let $B$ be a finite Blaschke product of degree $N$ with distinct zeros $a_1,\dots, a_N$,
\begin{equation}\label{bla}
B(z)=e^{i\gamma}\prod_{j=1}^{N}\frac{z-a_j}{1-\bar {a}_jz},
\end{equation}
where $\gamma\in\mathbb{R}$.
As usual, for $w\in \mathbb{D}$ by
$$k^B_w(z)=\frac{1-\overline{B(w)}B(z)}{1-\bar w z}$$
we denote the reproducing kernel for $K^2_B$, that is,
$$f(w)=\langle f, k^B_w\rangle$$ for $f\in K^2_B$. Note that for $j=1,\dots, N$ we have
\begin{equation}\label{kr}
k_j(z):=k^B_{a_j}(z)=\frac{1}{1-\bar a_j z}.
\end{equation}
As it was observed in \cite{CRW}, the model space $K^2_B$ is $N$--dimensional and the functions $k_1,\dots,k_N$ form a
(non--orthonormal) basis for $K^2_B$.

A simple computation gives the following.

\begin{lemma}[\cite{CRW}, p.5]\label{tech}
\

\begin{enumerate}[(1)]
\item $(C_B k_j)(z)=\frac{B(z)}{z-a_j}$ for $j=1,\dots, N$.
\item $\langle C_B k_j,k_i\rangle=\left\{\begin{array}{cc}0 &\text{for } i\ne j,\\ B'(a_j) &\text{for } i=j.\end{array}\right.$
\item $\langle k_j, k_i \rangle=\frac{1}{1-\bar a_ja_i}$.
\end{enumerate}
\end{lemma}

\begin{lemma}\label{cchar}
Let $B$ be a finite Blaschke product of degree $N$ with distinct zeros $a_1,\dots, a_N$. Let $C_B$ be the conjugation in $K^2_B$ given by $C_B f(z)=B(z)\overline{zf(z)}$ for $f\in K^2_B$. Assume that an operator $A\in L(K^2_B)$ has a matrix representation $(b_{i,j})_{i,j=1,\dots,N}$ with respect to the basis $\{k_1,\dots,k_N\}$. Then the following are equivalent:
 \begin{enumerate}[(1)]
 \item $A$ is $C_B$--symmetric;
 \item $\langle Ak_i, C_B k_j\rangle=\langle Ak_j,C_B k_i \rangle$ for all $i, j=1,\dots, N$;
 \item $\overline{B'(a_j)}b_{j,i}=\overline{B'(a_i)}b_{i,j}$ for all $i, j=1,\dots, N$.
 \end{enumerate}
\end{lemma}

\begin{proof}
The implication $(1)\Rightarrow (2)$ follows from
\begin{equation*}
\langle Ak_i, C_B k_j\rangle=\langle C_B^2 k_j,C_BAk_i\rangle=\langle k_j, A^* C_B k_i\rangle=\langle Ak_j,C_B k_i\rangle.
\end{equation*} The reverse implication can be proved similarly.

To prove that $(2)\Leftrightarrow (3)$ note that $Ak_i=\sum_{m=1}^N b_{m,i} k_m$. Hence, by Lemma \ref{tech}(2),
$$\langle Ak_i, C_B k_j\rangle=\sum_{m=1}^N b_{m,i}\langle k_m,C_B k_j\rangle=\overline{B'(a_j)} b_{j,i}.$$
Analogously,
$$\langle Ak_j, C_B k_i\rangle=\overline{B'(a_i)} b_{i,j}.$$
\end{proof}

Let $1\leq n\leq N$. Denote by $B_n$  the finite Blaschke product with $n$ distinct zeros $a_1,\dots, a_n$,
\begin{equation}\label{blafin}
B_n(z)=\prod_{j=1}^{n}\frac{z-a_j}{1-\bar {a}_jz},
\end{equation}
and by $C_{n}=C_{B_n}$ the conjugation in $K^2_{B_{n}}$ given by
$$(C_{n}f)(z)=B_n(z) \overline{zf(z)},\quad |z|=1.$$

\begin{theorem}\label{Blaschke}
  Let $B$ be a finite Blaschke product of degree $N$ with distinct zeros $a_1,\dots, a_N$. Denote by $B_n$ the Blaschke product of degree $n$ with zeros $a_1,\dots, a_n$ and by $P_n$ the orthogonal projection from $K^2_{B}$ onto $K^2_{B_n}$ for $n=1,\dots, N$. Let $A\in L(K^2_B)$. The following conditions are equivalent:
  \begin{enumerate}[(1)]
    \item $A\in \mathcal{T}(B)$;
    \item for every Blaschke product $B_\sigma$ dividing $B$ the operator $A_\sigma=P_{B_\sigma} A_{|K^2_{B_\sigma}}$ is $C_{B_\sigma}$--symmetric;
    \item for every $n=1,\dots, N$ the operator $A_n=P_n A_{|K_{B_n}^2}$ is $C_{n}$--symmetric.
  \end{enumerate}
\end{theorem}

To give the proof of Theorem \ref{Blaschke} we need two technical lemmas.
Firstly, let us observe by \eqref{kr} that $k^{B_n}_{a_j}= k^{B}_{a_j}=k_{j}$ for $1\leqslant n\leqslant N$, $j=1,\dots,n$. Hence  $\{k_1,\dots,k_n\}$ is a basis for $K^2_{B_n}\subset K^2_{B}$.

\begin{lemma}\label{rzut} For $1\leqslant m,n\leqslant N$ the following holds:
\begin{enumerate}[(1)]
\item $\langle C_nk_j,k_m\rangle=\left\{\begin{array}{cl}
0&\text{for } m\leq n, m\neq j,\\
B_n'(a_j)&\text{for } m\leq n, m= j,\\
\frac{B_n(a_m)}{a_m-a_j}&\text{for } m> n,
\end{array}\right.$ \qquad for  $j=1,\dots,n$;

\item $P_{n}k_m=\sum\limits_{j=1}^{n}\frac{\overline{{B_{n}(a_m)}}}{\overline{B_{n}^{\prime}(a_j)}(\bar a_m-\bar a_j)}\ k_j$ \  for  $n<m$;
\item $\frac{B_{n-1}(a_n)}{B_n'(a_n)}=1-|a_n|^2$ \  for $n>1$;
\item $\frac{B'_{n-1}(a_j)}{B_n'(a_j)}=\frac{1-\bar a_n a_j}{a_j-a_n}$ \  for $n>1$, $j=1,\dots,n-1$.
\end{enumerate}
\end{lemma}
\begin{proof}
To show (1) note that $C_nk_j\in K^2_{B_n}\subset K^2_B$ for $1\leq n\leq N$, $j=1,\ldots,n$, and that
$$(C_nk_j)(z)=\frac{B_n(z)}{z-a_j}$$
by Lemma \ref{tech}(1). If $m>n$, then the reproducing property of $k_m$ yields
$$\langle C_nk_j,k_m\rangle=(C_nk_j)(a_m)=\frac{B_n(a_m)}{a_m-a_j}.$$
On the other hand, if $m\leq n$, then it follows from Lemma \ref{tech}(2) that
\begin{displaymath}
\langle C_nk_j,k_m\rangle=\left\{\begin{array}{cc}
0&\mathrm{for}\ m\neq j,\\
B_n'(a_j)&\mathrm{for}\ m= j.
\end{array}\right.
\end{displaymath}

To show (2) assume that $m>n$ and $P_{n}k_m=\sum\limits_{l=1}^{n}d_lk_l$. Then, by part (1), for $j=1,\ldots,n$,
\begin{displaymath}
\frac{B_n(a_m)}{a_m-a_j}=\langle C_{n}k_j,k_m\rangle=\langle C_{n}k_j,P_nk_m\rangle=\sum_{l=1}^{n}\bar d_l\langle C_{n}k_j, k_l\rangle=B'_{n}(a_j)\bar d_j.
\end{displaymath}
Hence
$$d_j=\frac{\overline{{B_{n}(a_m)}}}{\overline{B_{n}^{\prime}( a_j)}(\bar a_m-\bar a_j)},$$ which proves (2).
The statements (3) and (4) follow directly from
$$B_n'(z)=B_{n-1}'(z)\frac{z-a_n}{1-\bar a_n z}+B_{n-1}(z)\frac{1-|a_n|^2}{(1-\bar a_nz)^2}.$$
\end{proof}

\begin{lemma}\label{macierz}
Let $A\in L(K^2_{B_n})$ have a matrix representation $(b_{i,j}^{(n)})_{i,j=1,\dots,n}$ with respect to the basis $\{k_1,\dots,k_n\}$. Then $A_{n-1}=P_{n-1}A_{|K^2_{B_{n-1}}}$ has a matrix representation $(b_{i,j}^{(n-1)})_{i,j=1,\dots,n-1}$, $$b_{i,j}^{(n-1)}=b_{i,j}^{(n)}+\frac{\overline{B_{n-1}(a_n)}b_{n,j}^{(n)}}{\overline{B'_{n-1}(a_i)}(\bar a_n-\bar a_i)},$$ with respect to the basis $\{k_1,\dots,k_{n-1}\}$.
\end{lemma}

\begin{proof}
Note that by Lemma \ref{rzut}(2),
$$P_{n-1}k_n=\sum\limits_{m=1}^{n-1}\frac{\overline{{B_{n-1}(a_n)}}}{\overline{B_{n-1}^{\prime}(a_m)}(\bar a_n-\bar a_m)}\ k_m.$$
Hence, for $j=1,\ldots,n-1$, we have
\begin{displaymath}
\begin{split}
P_{n-1}(A k_j)&=P_{n-1}\left(\sum_{m=1}^{n}b_{m,j}^{(n)}k_m\right)=P_{n-1}\left(\sum_{m=1}^{n-1}b_{m,j}^{(n)}k_m\right)+P_{n-1}b_{n,j}^{(n)}k_n\\
&=\sum_{m=1}^{n-1}\left(b_{m,j}^{(n)}+\tfrac{\overline{{B_{n-1}(a_n)}}b_{n,j}^{(n)}}{\overline{B_{n-1}^{\prime}( a_m)}(\bar a_n-\bar a_m)}\right)k_m.
\end{split}
\end{displaymath}
Since
$$b_{i,j}^{(n-1)}=\frac{1}{\overline{B_{n-1}'(a_i)}}\langle P_{n-1}(Ak_j), C_{n-1}k_i\rangle,\quad 1\leq i,j\leq n-1,$$
we get
\begin{displaymath}
\begin{split}
b_{i,j}^{(n-1)}&=\frac{1}{\overline{B_{n-1}'(a_i)}}\sum_{m=1}^{n-1}\left(b_{m,j}^{(n)}+\tfrac{\overline{{B_{n-1}(a_n)}}b_{n,j}^{(n)}}{\overline{B_{n-1}^{\prime}( a_i)}(\bar a_n-\bar a_m)}\right)\langle k_m, C_{n-1}k_i\rangle\\
&=b_{i,j}^{(n)}+\tfrac{\overline{{B_{n-1}(a_n)}}b_{n,j}^{(n)}}{\overline{B_{n-1}^{\prime}( a_i)}(\bar a_n-\bar a_i)}
\end{split}
\end{displaymath}
by Lemma \ref{rzut}(1).
\end{proof}

\begin{proof}[Proof of Theorem \ref{Blaschke}]
Since multiplying $B$ by a constant of modulus $1$ does not change $K_B^2$, we can assume without any loss of generality that $B$ is given by \eqref{bla} with $\gamma=0$, that is, $B=B_N$.

The implication $(1)\Rightarrow(2)$ follows form Lemma \ref{obciecie} and the implication $(2)\Rightarrow(3)$ is obvious. We only need to prove the implication $(3)\Rightarrow (1)$.
This  will be proved by induction. Note firstly that it is true for $N=2$ by \cite[p. 505]{Sarason}.

Assume now that the assertion is true for $n-1<N$, which means that $A_{n-1}=P_{n-1}A_{|K^2_{B_{n-1}}}$ is Toeplitz and has a matrix representation   $( b_{i,j}^{(n-1)})_{i,j=1,\dots,n-1}$ with respect to the basis $\{k_1,\dots,k_{n-1}\}$ satisfying
\begin{equation}\label{toepl}
b_{i,j}^{(n-1)}=\frac{\overline{B'_{n-1}(a_1)}}{\overline{B'_{n-1}(a_i)}}\left(\frac{ b_{1,i}^{(n-1)}(\bar a_1-\bar a_i)+ b_{1,j}^{(n-1)}(\bar a_j-\bar a_1)}{\bar a_j-\bar a_i}\right),
\end{equation}
 for  $ 1\leqslant i, j \leqslant n-1$, $i\ne j$, by \cite[Theorem 1.4]{CRW}. Assume also that $A\in L(K^2_{B_n})$ is $C_{n}$--symmetric and has a matrix representation $(b_{i,j}^{(n)})_{i,j=1,\dots,n}$ with respect to the basis $\{k_1,\dots,k_n\}$. We will show that $A$ is Toeplitz, i.e.,  $b_{i,j}^{(n)}$ satisfies  \cite[Theorem 1.4]{CRW}:
 \begin{equation*}
b_{i,j}^{(n)}=\frac{\overline{B'_{n}(a_1)}}{\overline{B'_{n}(a_i)}}\left(\frac{ b_{1,i}^{(n)}(\bar a_1-\bar a_i)+ b_{1,j}^{(n)}(\bar a_j-\bar a_1)}{\bar a_j-\bar a_i}\right),
\end{equation*}
 for  $ 1\leqslant i, j \leqslant n$, $i\ne j$.

 Since $A_{n-1}$ is $C_{n-1}$--symmetric, for $i, j=1,\dots,n-1$ we have, by Lemma \ref{cchar} and Lemma \ref{macierz},
 \begin{equation}\label{eq11}
 b_{j,i}^{(n-1)}=\frac{\overline{B'_{n-1}(a_i)}}{\overline{B'_{n-1}(a_j)}}b_{i,j}^{(n-1)}
=\frac{\overline{B'_{n-1}(a_i)}}{\overline{B'_{n-1}(a_j)}}\left(b_{i,j}^{(n)}+\frac{\overline{B_{n-1}(a_n)}b_{n,j}^{(n)}}{\overline{B'_{n-1}(a_i)}(\bar a_n-\bar a_i)}\right).
\end{equation}
 On the other hand, by Lemma \ref{macierz} and using the $C_n$--symmetry of $A$,
 \begin{equation}\label{eq12}
 \begin{split}
b_{j,i}^{(n-1)}&=b_{j,i}^{(n)}+\frac{\overline{B_{n-1}(a_n)}}{\overline{B'_{n-1}(a_j)}} \frac{b_{n,i}^{(n)}}{\bar a_n-\bar a_j}\\
&=\frac{\overline{B_n'(a_i)}}{\overline{B_n'(a_j)}}b_{i,j}^{(n)}+\frac{\overline{B_n'(a_i)}}{\overline{B_n'(a_n)}}\frac{\overline{B_{n-1}(a_n)}}{\overline{B'_{n-1}(a_j)}} \frac{b_{i,n}^{(n)}}{\bar a_n-\bar a_j}.
\end{split}
 \end{equation}
 Comparing \eqref{eq11} with \eqref{eq12} and putting $i=1$ we obtain
\[
 \frac{\overline{B_{n-1}(a_n)}}{\overline{B'_{n-1}(a_j)}}\frac{b_{n,j}^{(n)}}{\bar a_n-\bar a_1}
 =\frac{\overline{B_n'(a_1)}}{\overline{B_n'(a_n)}}\frac{\overline{B_{n-1}(a_n)}}{\overline{B'_{n-1}(a_j)}}\frac{b_{1,n}^{(n)}}{\bar a_n-\bar a_j}+\left(\frac{\overline{B_n'(a_1)}}{\overline{B_n'(a_j)}}-\frac{\overline{B'_{n-1}(a_1)}}{\overline{B'_{n-1}(a_j)}}\right)b_{1,j}^{(n)}.
\]
 Hence
 \begin{equation}\label{czdz}
 b_{n,j}^{(n)}=\frac{\overline{B_n'(a_1)}}{\overline{B_n'(a_n)}}\frac{b_{1,n}^{(n)}(\bar a_n-\bar a_1)}{\bar a_n-\bar a_j}
 +\left(\frac{\overline{B_n'(a_1)}}{\overline{B_{n-1}(a_n)}}\frac{\overline{B'_{n-1}(a_j)}}{\overline{B_n'(a_j)}}-\frac{\overline{B'_{n-1}(a_1)}}{\overline{B_{n-1}(a_n)}} \right)b_{1,j}^{(n)}(\bar a_n-\bar a_1).
 \end{equation}
 Using Lemma \ref{rzut} we can simplify
\begin{displaymath}
\begin{split}
\frac{\overline{B_n'(a_1)}}{\overline{B_{n-1}(a_n)}}&\frac{\overline{B'_{n-1}(a_j)}}{\overline{B_n'(a_j)}}-\frac{\overline{B'_{n-1}(a_1)}}{\overline{B_{n-1}(a_n)}}\\
&= \frac{\overline{B_n'(a_1)}}{\overline{B_n'(a_n)}}\frac{1}{1-|a_n|^2}\left(\frac{1-a_n\bar a_j}{\bar a_j-\bar a_n}+\frac{1-a_n\bar a_1}{\bar a_n-\bar a_1} \right)\\
&=  \frac{\overline{B_n'(a_1)}}{\overline{B_n'(a_n)}}\frac{(\bar a_j-\bar a_1)}{(\bar a_j-\bar a_n)(\bar a_n-\bar a_1)},
\end{split}
\end{displaymath}
  which together with \eqref{czdz} gives
 \begin{equation}\label{ntoep}
 b_{n,j}^{(n)}= \frac{\overline{B_n'(a_1)}}{\overline{B_n'(a_n)}}\left(\frac{b_{1,n}^{(n)}(\bar a_1-\bar a_n)+b_{1,j}^{(n)}(\bar a_j-\bar a_1)}{\bar a_j-\bar a_n}\right)
 \end{equation}
for $1\leqslant j\leqslant n-1$.
From \eqref{ntoep}, the $C_n$--symmetry of $A_n$ and Lemma \ref{cchar} we also get
 \begin{equation}\label{ntoep2}
 b_{i,n}^{(n)}= \frac{\overline{B_n'(a_n)}}{\overline{B_n'(a_i)}}b_{n,i}^{(n)}= \frac{\overline{B_n'(a_1)}}{\overline{B_n'(a_i)}}\left(\frac{b_{1,i}^{(n)}(\bar a_1-\bar a_i)+b_{1,n}^{(n)}(\bar a_n-\bar a_1)}{\bar a_n-\bar a_i}\right)
 \end{equation}
for $1\leqslant i\leqslant n-1$.
 By Lemma \ref{cchar} and Lemma \ref{macierz}, we have for $i=1,\ldots, n-1 $,
\begin{equation*}
b_{1,i}^{(n-1)}=\frac{\overline{B'_{n-1}(a_i)}}{\overline{B'_{n-1}(a_1)}}b_{i,1}^{(n-1)}=\frac{\overline{B'_{n-1}(a_i)}}{\overline{B'_{n-1}(a_1)}}\left(b_{i,1}^{(n)}+\frac{\overline{B_{n-1}(a_n)}}{\overline{B'_{n-1}(a_i)}}\frac{b_{n,1}^{(n)}}{\bar a_n-\bar a_i} \right).
\end{equation*}
Using Lemma \ref{cchar} again we obtain
\begin{equation}\label{eq42}
b_{1,i}^{(n-1)}=\frac{\overline{B_n'(a_1)}}{\overline{B_n'(a_i)}}\frac{\overline{B'_{n-1}(a_i)}}{\overline{B'_{n-1}(a_1)}}b_{1,i}^{(n)}+\frac{\overline{B_n'(a_1)}}{\overline{B_n'(a_n)}}\frac{\overline{B_{n-1}(a_n)}}{\overline{B'_{n-1}(a_1)}}\frac{b_{1,n}^{(n)}}{\bar a_n-\bar a_i}.
\end{equation}
Now applying Lemma \ref{macierz} to the left--hand side of \eqref{toepl}, and formula \eqref{eq42} to the right--hand side of \eqref{toepl} we can calculate for all $i,j=1,\dots,n-1$,
\begin{align}\label{eq43}
b_{i,j}^{(n)}&+\frac{\overline{B_{n-1}(a_n)}}{\overline{B'_{n-1}(a_i)}}\frac{b_{n,j}^{(n)}}{\bar a_n-\bar a_i}\notag\\
&=\frac{\overline{B_n'(a_1)}}{\overline{B_n'(a_i)}}\frac{\bar a_1-\bar a_i}{\bar a_j-\bar a_i}b_{1,i}^{(n)}+\frac{\overline{B_n'(a_1)}}{\overline{B_n'(a_n)}}\frac{\overline{B_{n-1}(a_n)}}{\overline{B'_{n-1}(a_i)}}\frac{(\bar a_1-\bar a_i)b_{1,n}^{(n)}}{(\bar a_n-\bar a_i)(\bar a_j-\bar a_i)}\\
&\quad+\frac{\overline{B_n'(a_1)}}{\overline{B_n'(a_j)}}\frac{\overline{B'_{n-1}(a_j)}}{\overline{B'_{n-1}(a_i)}}\frac{\bar a_j-\bar a_1}{\bar a_j-\bar a_i}b_{1,j}^{(n)}
+\frac{\overline{B_n'(a_1)}}{\overline{B_n'(a_n)}}\frac{\overline{B_{n-1}(a_n)}}{\overline{B'_{n-1}(a_i)}}\frac{\bar a_j-\bar a_1}{(\bar a_n-\bar a_j)(\bar a_j-\bar a_i)}b_{1,n}^{(n)}.\notag
\end{align}
Note that
\begin{multline*}
\frac{\overline{B_n'(a_1)}}{\overline{B_n'(a_j)}}\frac{\overline{B'_{n-1}(a_j)}}{\overline{B'_{n-1}(a_i)}}\frac{\bar a_j-\bar a_1}{\bar a_j-\bar a_i}b_{1,j}^{(n)}\\
=\frac{\overline{B_n'(a_1)}}{\overline{B_n'(a_i)}}\frac{\bar a_j-\bar a_1}{\bar a_j-\bar a_i}b_{1,j}^{(n)}+\frac{\overline{B_{n-1}(a_n)}}{\overline{B'_{n-1}(a_i)}}\frac{\overline{B_n'(a_1)}}{\overline{B_n'(a_n)}}\frac{\bar a_j-\bar a_1}{(\bar a_n-\bar a_i)(\bar a_j-\bar a_n)}b_{1,j}^{(n)}
\end{multline*}
by Lemma \ref{rzut}. Moreover,
\begin{equation}\label{eq44}
\frac{1}{\bar a_j-\bar a_i}\left(\frac{\bar a_1-\bar a_i}{\bar a_n-\bar a_i}+\frac{\bar a_j-\bar a_1}{\bar a_n-\bar a_j} \right)=\frac{\bar a_n-\bar a_1}{(\bar a_n-\bar a_i)(\bar a_n-\bar a_j)}.
\end{equation}
Hence, \eqref{eq43} and \eqref{eq44} give
\begin{equation*}
\begin{split}
b_{i,j}^{(n)}&+\frac{\overline{B_{n-1}(a_n)}}{\overline{B'_{n-1}(a_i)}}\frac{b_{n,j}^{(n)}}{\bar a_n-\bar a_i}
=\frac{\overline{B_n'(a_1)}}{\overline{B_n'(a_i)}}\frac{b_{1,i}^{(n)}(\bar a_1-\bar a_i)+b_{1,j}^{(n)}(\bar a_j-\bar a_1)}{\bar a_j-\bar a_i}\\
&+\frac{\overline{B_{n-1}(a_n)}}{\overline{B'_{n-1}(a_i)}}\frac{1}{\bar a_n-\bar a_i}\frac{\overline{B_n'(a_1)}}{\overline{B_n'(a_n)}}\frac{b_{1,n}^{(n)}(\bar a_1-\bar a_n)+b_{1,j}^{(n)}(\bar a_j-\bar a_1)}{\bar a_j-\bar a_n}.
\end{split}
\end{equation*}
Taking into account \eqref{ntoep} and \eqref{ntoep2}, the above equation implies that
$$b_{i,j}^{(n)}=\frac{\overline{B_n'(a_1)}}{\overline{B_n'(a_i)}}\left(\frac{b_{1,i}^{(n)}(\bar a_1-\bar a_i)+b_{1,j}^{(n)}(\bar a_j-\bar a_1)}{\bar a_j-\bar a_i}\right)$$
for all $1\leqslant i,j\leqslant n$, $i\ne j$, which completes the proof.
\end{proof}

\section{An infinite Blaschke product with uniformly separated zeros}

Let $B$ be an infinite Blaschke product,
\begin{equation}\label{inff}
B(z)=e^{i\gamma}\prod_{j=1}^{\infty}\frac{\bar {a}_j}{|a_j|}\frac{a_j-z}{1-\bar {a}_jz},\quad \gamma\in\mathbb{R},
\end{equation}
 (if $a_j=0$, then $\bar {a}_j/|a_j|$ is interpreted as $-1$) with uniformly separated zeros $a_1,a_2,\ldots$, i.e.,
\begin{equation}\label{delta}
\inf_{n}\prod_{j\neq n}\left|\frac{a_j-a_n}{1-\bar {a}_ja_n}\right|\geq \delta
\end{equation}
for some $\delta>0$. In particular, the zeros $\{a_j\}_{j=1}^{\infty}$ are distinct. As before, $B_n$, $n\in\mathbb{N}$, denotes the finite Blaschke product with zeros $a_1,\ldots,a_n$, given by \eqref{blafin}.

\begin{theorem}\label{binf}
Let $B$ be an infinite Blaschke product with uniformly separated
zeros $\{a_j\}_{j=1}^{\infty}$. Denote by $B_n$ the Blaschke product of degree $n$ with distinct zeros $\{a_1,\ldots,a_n\}$ and by $P_n$ the orthogonal projection form $K^2_B$ onto $K^2_{B_n}$ for $n\in\mathbb{N}$. Let $A\in L(K^2_B)$. The following conditions are equivalent:
\begin{itemize}
\item[(1)] $A\in\mathcal{T}(B)$;
\item[(2)] for every Blaschke product $B_{\sigma}$ dividing $B$ the operator $A_{\sigma}=P_{B_{\sigma}}A_{|K^2_{B_{\sigma}}}$ is $C_{B_{\sigma}}$--symmetric;
\item[(3)] for every $n\in\mathbb{N}$ the operator $A_{n}=P_{n}A_{|K^2_{B_{n}}}$ is $C_{n}$--symmetric.
\end{itemize}
\end{theorem}

Again, before we give the proof some preparations are necessary.
Clearly, $K^2_{B_n}\subset K_B^2$ for all $n\in\mathbb{N}$ and $k_{a_j}^B=k_j$ for all $j\in\mathbb{N}$. Condition \eqref{delta} implies that the reproducing kernels $k_j$, $j\in\mathbb{N}$, form a basis for $K^2_B$ (for more details see \cite[Chapter 12]{GMR}, \cite{GP} or \cite{Nikolski}). In particular, every $f\in K^2_B$ can be written as
$$f=\sum_{j=1}^{\infty}\frac{\langle f,C_B k_j \rangle}{\overline{B'(a_j)}}k_j,$$
where the series converges in the  norm.

\begin{lemma}\label{bij}
Let $A\in L(K_B^2)$ have a matrix representation $(b_{i,j})_{i,j=1}^{\infty}$ with respect to the basis $\{{k}_{i}\colon i\in\mathbb{N}\}$. Then $A_{n}=P_{n}A_{|K^2_{B_{n}}}$ has a matrix representation $(b_{i,j}^{(n)})_{i,j=1,\ldots,n}$,
\begin{displaymath}
b_{i,j}^{(n)}=b_{i,j}+\sum_{m=n+1}^{\infty}\frac{\overline{B_n(a_m)}b_{m,j}}{\overline{B_n'(a_i)}(\bar {a}_m-\bar {a}_i)},
\end{displaymath}
with respect to the basis $\{{k}_{1},\ldots, {k}_{n}\}$.
\end{lemma}
\begin{proof}
Let $n\in\mathbb{N}$ and $1\leq i,j\leq n$. Since
$$A_nk_j=\sum_{m=1}^{n}b_{m,j}^{(n)}k_m,$$
Lemma \ref{rzut}(1) gives
$$b_{i,j}^{(n)}=\frac{1}{\overline{B_n'(a_i)}}\langle A_nk_j, C_nk_i\rangle.$$
Since
$$Ak_j=\sum_{m=1}^{\infty}b_{m,j}k_m,$$
and the series converges in norm, we get
\begin{displaymath}
\begin{split}
b_{i,j}^{(n)}
&=\frac{1}{\overline{B_n'(a_i)}}\langle Ak_j, C_nk_i\rangle
=\frac{1}{\overline{B_n'(a_i)}}\sum_{m=1}^{\infty}b_{m,j}\langle k_m, C_nk_i\rangle\\
&=b_{i,j}+\frac{1}{\overline{B_n'(a_i)}}\sum_{m=n+1}^{\infty}\frac{\overline{B_n(a_m)}}{\bar {a}_m-\bar {a}_i}b_{m,j}
\end{split}
\end{displaymath}
by Lemma \ref{rzut}(1).
\end{proof}

\begin{corollary}\label{granica}
For all $i,j\in\mathbb{N}$,
\begin{displaymath}
b_{i,j}=\lim_{n\rightarrow\infty}b_{i,j}^{(n)}.
\end{displaymath}
\end{corollary}
\begin{proof}
It is known that the infinite Blaschke product $B$ converges uniformly on compact subsets of $\mathbb{D}$. It follows that if $$\lambda_n=(-1)^n\prod_{j=1}^{n}\frac{\bar {a}_j}{|a_j|},\quad n\in\mathbb{N},$$
then $\lambda_nB_n\rightarrow B$ and $\lambda_nB_n'\rightarrow B'$ as $n\rightarrow\infty$ (uniformly on compact subsets of $\mathbb{D}$). In particular,
$$\lambda_nB_n'(a_i)\rightarrow B'(a_i)\quad\mathrm{as}\quad n\rightarrow\infty$$
for each $i\in\mathbb{N}$.
Fix $i,j\in\mathbb{N}$. Let $n\geq \max\{i,j\}$ and write
$$Ak_j=\sum_{m=1}^{n}b_{m,j}k_m+r_n, \quad \text{ where }\quad r_n=\sum_{m=n+1}^{\infty}b_{m,j}k_m.$$
As in the proof of Lemma \ref{bij},
\begin{displaymath}
\begin{split}
b_{i,j}^{(n)}&
=b_{i,j}+\tfrac{1}{\overline{B_n'(a_i)}}\langle r_n, C_nk_i\rangle
=b_{i,j}+\tfrac{1}{\overline{\lambda_nB_n'(a_i)}}\langle r_n, \lambda_nC_nk_i\rangle,
\end{split}
\end{displaymath}
where the last equality follows form the fact that $\lambda_n\in\mathbb{T}$.
Since $r_n$ tends to zero in the norm, the sequence $(\lambda_nC_nk_i)_{n\geq i}$ is bounded and $\lambda_nB_n'(a_i)\rightarrow B'(a_i)$, we get
\begin{equation*}\lim_{n\rightarrow\infty}b_{i,j}^{(n)}=b_{i,j}+\lim_{n\rightarrow\infty}\left(\tfrac{1}{\overline{\lambda_nB_n'(a_i)}}\langle r_n, \lambda_nC_nk_i\rangle\right)= b_{i,j}.\qedhere\end{equation*}
\end{proof}
\begin{proof}[Proof of Theorem \ref{binf}]
As in the proof of Theorem \ref{Blaschke}, without loss of generality, assume that $B$ is given by \eqref{inff} with $\gamma=0$.
The implication $(1)\Rightarrow(2)$ follows from Lemma \ref{obciecie} and the implication $(2)\Rightarrow(3)$ is obvious. We only need to prove $(3)\Rightarrow(1)$.

Let $A\in L(K^2_B)$ and assume that $A_{n}=P_{n}A_{|K^2_{B_{n}}}$ is $C_{n}$--symmetric for every $n\in\mathbb{N}$. By \cite[Remark 2.4]{lanucha}, to prove that $A\in\mathcal{T}(B)$ it is enough to show that
\begin{equation}\label{toep01}
b_{i,j}=\frac{\overline{B'(a_1)}}{\overline{B'(a_i)}}\left(\frac{b_{1,i}(\bar {a}_1-\bar {a}_i)+b_{1,j}(\bar {a}_j-\bar {a}_1)}{\bar {a}_j-\bar {a}_i}\right)
\end{equation}
for all $i\neq j$, where $(b_{i,j})_{i,j=1}^{\infty}$ is the matrix representation of $A$ with respect to the basis $\{{k}_{i}\colon i\in\mathbb{N}\}$.
Fix $i,j\in\mathbb{N}$, $i\neq j$, and take an arbitrary $N\geq\max\{i,j\}$. By (3), $P_{n}A_{N|K^2_{B_{n}}}=A_n$ is $C_{n}$--symmetric for all $n=1,\ldots,N$. Hence Theorem \ref{Blaschke} implies that $A_N\in\mathcal{T}(B_N)$. By \cite[Theorem 4.1]{CRW},
\begin{equation}\label{toep02}
b_{i,j}^{(N)}=\frac{\overline{B_N'(a_1)}}{\overline{B_N'(a_i)}}\left(\frac{b_{1,i}^{(N)}(\bar {a}_1-\bar {a}_i)+b_{1,j}^{(N)}(\bar {a}_j-\bar {a}_1)}{\bar {a}_j-\bar {a}_i}\right),
\end{equation}
where $(b_{i,j}^{(N)})_{i,j=1,\ldots,N}$ is the matrix representation of the operator $A_N$ with respect to the basis $\{{k}_{1},\ldots,k_N\}$. Taking the limit in \eqref{toep02} as $N$ tends to infinity we get \eqref{toep01} because $b_{i,j}^{(N)}\rightarrow b_{i,j}$ and $B_{N}'(a_i)\rightarrow B'(a_i)$ by Corollary \ref{granica} and its proof.
\end{proof}

\section{Conjecture}
Theorems \ref{toep}, \ref{Blaschke}, \ref{binf} and Proposition \ref{b1} suggest that the following conjecture can be true:
\begin{conjecture}
Let $\theta$ be a nonconstant inner function, and let $A\in L(K^2_\theta)$. Then $A\in \mathcal{T}(\theta)$ if and only if for every nonconstant inner function $\alpha$ dividing $\theta$ the operator $A_\alpha=P_\alpha A_{|\kda}$ is $C_\alpha$--symmetric.
\end{conjecture}

The following example supports the conjecture.

\begin{example}
 Consider $$B(z)=z^2\frac{w-z}{1-\overline{w}z},\quad \mathrm{where}\quad w\neq 0.$$ Then the space $K_B^2$ has dimension $3$ and the set $\{1,z,\frac{z^2}{\|k_w\|}k_w\}$, $k_w(z)=(1-\overline{w}z)^{-1}$, is an orthonormal basis for $K_B^2$.

We first describe the operators form $\mathcal{T}(B)$ in terms of their matrix representations with respect to the basis $\{1,z,\frac{z^2}{\|k_w\|}k_w\}$.
Let $A_{\varphi}^B$, $\varphi\in L^2$, be an operator from $\mathcal{T}(B)$, and let $M_{A_{\varphi}^B}=(b_{i,j})$ be its matrix representation. By \cite[Theorem 3.1]{Sarason} we can assume that $\varphi\in \overline{BH^2}+BH^2$, namely, that
$$\varphi=c_{-2}{\tfrac{\bar z^2}{\|k_w\|}\bar k_w}+c_{-1}\overline{z}+c_{0}+c_{1}z+c_{2}\tfrac{z^2}{\|k_w\|}k_w.$$
It is now a matter of a simple computation to see that the matrix $M_{A_{\varphi}^B}=(b_{i,j})$ is given by
$$\left(\begin{array}{ccc}
c_{0}    &c_{-1}   &c_{-2}\\
c_{1}    &c_{0}   &c_{-2}\overline{w}+\frac{c_{-1}}{\|k_w\|}\\
c_{2}    &\,\ \,\frac{c_{1}}{\|k_w\|}+c_{2}w   &\,\ \,c_{-2}\overline{w}^2\|k_w\|+c_{-1}\overline{w}+c_{0}+c_{1}w+c_{2}w^2\|k_w\|
\end{array}\right).$$
From this, the elements $b_{i,j}$ are described by the following system of equations
\begin{align}\label{eqo1}
b_{2,2}&=b_{1,1}\\ \label{eqo2}
b_{2,3}&=\overline{w}b_{1,3}+\|k_w\|^{-1}b_{1,2}\\ \label{eqo3}
b_{3,2}&=\|k_w\|^{-1}b_{2,1}+wb_{3,1}\\  \label{eqo4}
b_{3,3}&=b_{1,1}+\overline{w}\|k_w\|b_{2,3}+w\|k_w\|b_{3,2}\\
\phantom{b_{3,3}}&=b_{1,1}+\overline{w}^2\|k_w\|b_{1,3}+\overline{w}b_{1,2}+wb_{2,1}+w^2\|k_w\|b_{3,1}.\notag
\end{align}
Clearly, each $3\times3$ matrix $(b_{i,j})$ satisfying \eqref{eqo1}--\eqref{eqo4} is determined by five elements (the first row and the first column) and the space $\mathcal{M}_B$ of all such matrices has dimension $5$. As matrices representing operators from $\mathcal{T}(B)$ have to belong to $\mathcal{M}_B$ and the dimension of $\mathcal{T}(B)$ in this case is also $5$, we conclude that a linear operator $A$ form $K_B^2$ into $K_B^2$ belongs to $\mathcal{T}(B)$ if and only if its matrix representation with respect to $\{1,z,\frac{z^2}{\|k_w\|}k_w\}$ satisfies \eqref{eqo1}--\eqref{eqo4}.

Now let $A$ be an operator from $K_B^2$ into $K^2_B$ such that for every $B_{\sigma}\leq B$ the compression $A_{\sigma}=P_{B_{\sigma}}A_{|B_{\sigma}}$ is $C_{B_{\sigma}}$-symmetric. Using the above characterization we show that $A$ must belong to $\mathcal{T}(B)$.
Let $M_A=(b_{i,j})$ be the matrix representation of $A$ with respect to the basis $\{1,z,\frac{z^2}{\|k_w\|}k_w\}$. Our goal is to show that $(b_{i,j})$ satisfies \eqref{eqo1}--\eqref{eqo4}.


Let $B_1(z)=z^2$ and $A_{1}=P_{B_{1}}A_{|B_{1}}$. Then the space $K^2_{B_1}$ is spanned by $\{1,z\}$,
$$C_{B_1}1=z,\quad C_{B_1}z=1,$$
and the $C_{B_1}$-symmetry of $A_1$ gives \eqref{eqo1}.

Let $B_2(z)=z\frac{w-z}{1-\overline{w}z}$ and $A_{2}=P_{B_{2}}A_{|B_{2}}$. Then the space $K^2_{B_2}$ is spanned by $\{1,k_w\}$,
$$C_{B_2}1=\tfrac{w-z}{1-\overline{w}z},\quad \mathrm{and}\quad C_{B_2}k_w=-zk_w.$$
Moreover, we have
$$C_Bz=\tfrac{w-z}{1-\overline{w}z}=C_{B_2}1,\quad C_B(z^2k_w)=-k_w,$$
and
$$\overline{w}b_{1,3}+\|k_w\|^{-1}b_{1,2}= \|k_w\|^{-1}\langle A\left(\overline{w}z^2k_w+z\right), 1\rangle.$$
 Since
 $$\overline{w}z^2k_w+z=zk_w$$
and $A$, $A_2$ are symmetric with respect to $C_B$ and $C_{B_2}$, respectively,
we obtain \eqref{eqo2}. Namely,
\begin{displaymath}
\begin{split}
\overline{w}b_{1,3}+\|k_w\|^{-1}b_{1,2}&= \|k_w\|^{-1}\langle A\left(zk_w\right), 1\rangle=-\|k_w\|^{-1}\langle A_2C_{B_2}k_w, 1\rangle\\
&=-\|k_w\|^{-1}\langle C_{B_2}A_2^{*}k_w, 1\rangle=-\|k_w\|^{-1}\langle A_2C_{B_2}1,k_w\rangle\\
&=\|k_w\|^{-1}\langle AC_{B}z,C_B(z^2k_w)\rangle=\|k_w\|^{-1}\langle C_{B}A^{*}z,C_B(z^2k_w)\rangle\\
&=\|k_w\|^{-1}\langle A(z^2k_w),z\rangle=b_{2,3}.
\end{split}
\end{displaymath}
Similarly we can obtain \eqref{eqo3}.


To get \eqref{eqo4} firstly, by using $C_B$-symmetry of $A$, we have
\begin{displaymath}
\begin{split}
b_{3,3}&=\|k_w\|^{-2}\langle A(z^2k_w),z^2k_w\rangle\\
&=\|k_w\|^{-2}\langle AC_Bk_w,C_Bk_w\rangle=\|k_w\|^{-2}\langle Ak_w,k_w\rangle.
\end{split}
\end{displaymath}
From this
\begin{displaymath}
\begin{split}
b_{3,3}-b_{1,1}&=(1-|w|^2)\langle Ak_w,k_w\rangle-\langle A1,1\rangle=\langle A\left(1-\overline{w}\tfrac{w-z}{1-\overline{w}z}\right),k_w\rangle-\langle A1,1\rangle\\
&=\langle A1,k_w-1\rangle-\overline{w}\langle A\left(\tfrac{w-z}{1-\overline{w}z}\right),k_w\rangle
=w\langle A1,zk_w\rangle+\overline{w}\langle AC_Bz,C_B(z^2k_w)\rangle\\
&=-w\langle A_21,C_{B_2}k_w\rangle+\overline{w}\langle A(z^2k_w),z\rangle
=-w\langle A_2k_w,C_{B_2}1\rangle+\overline{w}\|k_w\|b_{2,3}\\
&=w\langle AC_B(z^2k_w),C_{B}z\rangle+\overline{w}\|k_w\|b_{2,3}
=w\langle Az,z^2k_w\rangle+\overline{w}\|k_w\|b_{2,3}\\
&=w\|k_w\|b_{3,2}+\overline{w}\|k_w\|b_{2,3},
\end{split}
\end{displaymath}
which completes the proof.

\end{example}

\end{document}